\documentclass[leqno, 11pts]{article}

\usepackage{graphicx}

\usepackage{stmaryrd}
\usepackage{amsmath,amssymb,amsfonts}
\usepackage{times}
\usepackage[T1]{fontenc}
\usepackage{graphicx}
\usepackage[english]{babel}
\usepackage{amsmath}
\usepackage{amsthm}
\usepackage{amssymb}
\usepackage{enumerate}
\usepackage{xspace}
\usepackage{euscript}
\usepackage{graphicx}
\usepackage{amscd}
\usepackage{epsfig}
\usepackage{tabularx}
\usepackage{color}
\usepackage{bm}
\usepackage{tikz} 

\newtheorem{theorem}{Theorem}[section]
\newtheorem{lemma}[theorem]{Lemma}

\theoremstyle{definition}
\newtheorem{definition}[theorem]{Definition}

\theoremstyle{remark}
\newtheorem{remark}[theorem]{Remark}

\numberwithin{equation}{section}
\begin{document}
\date{}

\title{Analysis of a finite element DtN method for scattering resonances of sound hard obstacles}

\author{Yingxia Xi \thanks{School of Mathematics and Statistics, Nanjing University of Science and Technology, Nanjing, 210094, China. ({\tt xiyingxia@njust.edu.cn})} 
\and Bo Gong \thanks{School of Mathematics, Statistics and Mechanics, Beijing University of Technology, Beijing, 100124, China. ({\tt  gongbo@bjut.edu.cn}).}
\and Jiguang Sun \thanks{Department of Mathematical Sciences, Michigan Technological University, Houghton, MI 49931, U.S.A. ({\tt  jiguangs@mtu.edu}).}}

\maketitle

\begin{abstract}
Scattering resonances have important applications in many areas of science and engineering. They are the replacement of discrete spectral data for problems on non-compact domains. In this paper, we consider the computation of scattering resonances defined on the exterior to a compact sound hard obstacle. The resonances are the eigenvalues of a holomorphic Fredholm operator function. We truncate the unbounded domain and impose the Dirichlet-to-Neumann (DtN) mapping. The problem is then discretized using the linear Lagrange element. Convergence of the resonances is proved using the abstract approximation theory for holomorphic Fredholm operator functions. The discretization leads to nonlinear algebraic eigenvalue problems, which are solved by the recently developed parallel spectral indicator methods. Numerical examples are presented for validation.

\end{abstract}

\section{Introduction}
Scattering resonances have important applications in many areas of science and engineering. They are the replacement of discrete spectral data for problems on non-compact domains. Different from eigenvalues which describe the energy of bound states, scattering resonances describe a decaying state in which energy can scatter to infinity \cite{SemyonMaciej,MaciejPhysicsGeometry}. Mathematically, scattering resonances are the poles of the meromorphic continuation of the scattering operator to the lower half-plane \cite{LaxPhillips1989}. Scattering poles carry physical information with the real parts being the rates of oscillations and imaginary parts the rates of decay. In this paper, we consider the numerical computation of acoustic scattering resonances, which play a key role in applications such as vibration and noise reduction, and design of acoustic functional devices and musical instruments \cite{Huang}.

Considerable efforts have been devoted to the analysis of scattering resonances. The results include the existence, upper and lower bound estimation of the counting function, and the distribution of scattering resonances (see the comprehensive
monograph \cite{SemyonMaciej}). There exist much fewer works on the numerical computation. Development of effective computational methods for scattering resonances inevitably faces several challenges.
\begin{itemize}
\item[(1)] Scattering resonances are the poles of the meromorphic continuation of the scattering operator to the lower half-plane. There are lack of tools for the direct computation of poles. (We are not distinguishing resonances and poles in this paper although their meanings can be different in other context.)

\item[(2)] The scattering problem is defined on the unbounded domain. If the domain is truncated, some suitable artificial boundary condition needs to be enforced. However, artificial boundary conditions that work well for the scattering problems might not work for scattering resonances.

\item[(3)] The scattering operator depends nonlinearly on the wave number. The convergence analysis of finite element methods for nonlinear problem is challenging. 

\item[(4)] Numerical methods for scattering resonances lead to nonlinear matrix eigenvalue problems. Effective nonlinear algebraic eigensolvers are required.
\end{itemize}

There exist some works on the computation of scattering resonances \cite{ JonathanMarletta, HohageNannen2009, KimPasciak2009MC, OlafUnger2017MMAS, NannenWess2018BIT, Halla2022SINUM, Araujo2021JCP, SunMa, XiLinSun}. In \cite{NannenWess2018BIT}, the unbounded domain is truncated using 
the perfectly matched layer (PML) and the dependence of spurious resonances on the discretization parameters and the complex scaling function was analyzed. In \cite{JonathanMarletta}, the scattering resonance problem was treated based on the internal and external DtN mappings. The nonlinear matrix-valued eigenvalue problem was solved by the gradient descent method and thus some resonances may be missed in the calculation process. In \cite{SunMa}, an accurate Nystr\"{o}m method was proposed to compute the scattering poles for sound soft obstacles. Recently, in \cite{XiLinSun}, a robust finite element DtN method was developed to solve a multiscale, penetrable, highly nonlinear electromagnetic scattering resonance problem. An interesting relevant work \cite{Grubisic2023JSC} considered the computation of the quasi-resonances and proposed a boundary element method to search for the minima of the smallest singular value of the matrix as a function of the wave number (see also the determination of resonant states in \cite{Lenoir1992}).

In contrast to the computational results for scattering resonances, the convergence analysis is rarely addressed. In this paper, we propose a finite element DtN method and prove an error estimate. Combining with the parallel spectral indicator method (SIM), we develop an effective numerical method for scattering poles of sound hard obstacles. The method can be extended to compute scattering resonances for other scattering problems.

We start with a reformulation of the scattering poles as the eigenvalues of a holomorphic Fredholm operator function. As for the scattering problem, we truncate the unbounded domain using an artificial boundary, on which we enforce the DtN mapping (see, for example, \cite{Hsiao2011}). The advantage of the DtN mapping is the equivalence of the truncated problem and the original problem. Numerical experiments indicate that the DtN mapping does not pollute the spectrum. For scattering problems, other techniques including the absorbing boundary conditions and PML have been used to truncate the unbounded domain approximating the Sommerfeld radiation condition. However, these boundary conditions might not work or produce spurious modes that significantly pollute the spectrum \cite{KimPasciak2009MC,NannenWess2018BIT, Araujo2021JCP}.

On the truncated domain, the linear Lagrange element is employed for discretization. We prove the convergence using the abstract spectral approximation theory of holomorphic Fredholm operator functions \cite{karma1996a, karma1996b} by taking into account both errors of the truncated DtN mapping and the finite element discretization. 

Finally, a recently developed parallel SIM is used to compute the eigenvalues in a given region on the lower half-plane \cite{XiSun2023}. SIM is based on contour integrals and was originally designed for non-Hermitian matrix eigenvalue problems \cite{HuangSun, HuangSunYang}. The parallel SIM is highly scalable and has been successfully used to compute several nonlinear eigenvalue problems \cite{Gong2022MC, SunMa, XiLinSun}.

The proposed method is flexible in the sense that non-constant coefficients or non-smooth boundaries can be treated easily. It  does not need initial guesses, which are required by Newton type methods or gradient descent methods. Moreover, the use of parallel SIM is efficient by computing multiple poles simultaneously. The combination of the FEM, DtN mapping and parallel SIM provides a practical tool to compute scattering resonances of various problems.

The rest of the paper is organized as follows. In Section~\ref{SR}, we introduce the scattering resonances for sound hard obstacles and the DtN mapping on the artificial boundary. The scattering resonances are the eigenvalues of a holomorphic Fredholm operator function. Section~\ref{FE} proposes a finite element method for discretization. We analyze the approximation properties of the related discrete operators. The convergence analysis is then carried out using the abstract approximation theory for holomorphic Fredholm operator functions. In Section~\ref{NE}, the parallel SIM is used to compute the resulting nonlinear algebraic eigenvalue problem. Numerical experiments are presented to demonstrate the effectiveness. We draw some conclusions and discuss future work in Section~\ref{CF}.

%---------------------------------------------------------------------
\section{Scattering Resonances of Sound Hard Obstacles}\label{SR}
Let $D \subseteq \mathbb R^2$ be a bounded simply connected Lipschitz domain and $D^c := \mathbb{R}^2 \setminus \overline{D}$, where $\overline{D}$ is the closure of $D$. Denote by $\Gamma:=\partial D$, the boundary of $D$. Let $\nu$ denote the outward unit normal on the boundary of a domain. The acoustic scattering problem for the sound hard $D$ is to  find $u$ satisfying

\begin{equation}\label{NeumannBCP}
\left\{
\begin{array}{rlll}
\Delta u + k^2 u &=& 0&{\rm in} \quad D^c,\\
\frac{\partial u}{\partial \nu}&=&g&{\rm on}\quad \Gamma,\\
\end{array}
\right.
\end{equation}
and the Sommerfeld radiation condition 
\begin{equation}\label{SRC}
\frac{\partial u}{\partial r}-iku=o(r^{-1/2}), \quad r=|x|\rightarrow \infty.%\frac{\partial u}{\partial r}-iku=o(r^{-(n-1)/2}),
\end{equation}
If $\Im(k)\geq0$, then the scattering problem has a unique solution $u$ in $H_{\text{loc}}^1(D^c)$. The solution operator ${\mathcal S}(k) : g\mapsto u$ is holomorphic on the upper half-plane and can be meromorphically extended to $\mathbb{C}\backslash{\mathbb{R}^-}$ with $\mathbb{R}^-$ being the set of non-positive real numbers \cite{Lenoir1992}. The scattering resonances are defined as the poles of ${\mathcal S}$ in $\{k\in\mathbb{C}: {\rm Im}(k)<0\}$ \cite{SemyonMaciej}.

A solution $u(x)$ to the Helmholtz equation whose domain of definition contains the exterior of some disk is called radiating if it satisfies the Sommerfeld radiation condition \eqref{SRC}. It is well known that, for ${\rm Im}(k) \ge 0$, the radiating condition \eqref{SRC} is equivalent to the series expansion of $u$ of the form
\begin{equation}\label{outgoingconditon}
u(x)=\sum_{n=-\infty}^{\infty}a_nH_n^{(1)}(kr)e^{in\theta},\ \ \ |x|>r_0,
\end{equation}
where $\theta=\arg(x)$, $r_0>0$ is a constant large enough, and $H_n^{(1)}$ is the first kind Hankel function of order $n$  \cite{Hsiao2011}.

The scattering problem is defined on the unbounded domain. To use the finite element method, one usually truncates $D^c$ and introduces some artificial boundary condition. Let $\Gamma_R$ be a circle centered at the origin with radius $R$ and $D$ inside. We denote by $\Omega$ the annular region between $\Gamma$ and $\Gamma_R$. In this paper, we employ the DtN mapping on $\Gamma_R$. Let the set of zeros of Hankel function $H_n^{1}(kR)$ be $Z_n$ and $Z=\bigcup_{n=0}^{\infty}Z_n$. Define $\Lambda = \mathbb{C}\setminus (\mathbb{R}^-\cup Z)$. Following \cite{Hsiao2011}, for $k\in \Lambda$, we define the DtN operator $T(k):H^{1/2}(\Gamma_R)\rightarrow H^{-1/2}(\Gamma_R)$ as 
\begin{equation}\label{DtN-mapping}
    T(k)\varphi = \sum_{n=0}^{+\infty}{\!}^{{}^{\,\scriptstyle{\prime}}}\frac{k}{\pi}\frac{{H^{(1)}_n}^{\prime}(kR)}{H^{(1)}_n(kR)}\int_0^{2\pi}\varphi(\phi)\cos(n(\theta-\phi))\text{d}\phi,
\end{equation}
where the summation with a prime $'$ means that the zero-th term is factored
by $1/2$. 

The DtN operator $T(k)$ is bounded. Futhermore, for real number $s\geq1/2$,
$T(k):H^{s}(\Gamma_R)\rightarrow H^{s-1}(\Gamma_R), k\in \Lambda$, satisfies
\begin{equation}\label{Tk}
\|T(k)\varphi\|_{H^{s-1}(\Gamma_R)}\leq C\|\varphi\|_{H^{s}(\Gamma_R)},
\end{equation}
where $C>0$ is a constant depending on $kR$ but not $\varphi$.

Using $T(k)$, we obtain the following problem defined on the bounded domain $\Omega$, which is equivalent to the scattering problem \eqref{NeumannBCP} and \eqref{SRC} for ${\rm Im}(k)\ge 0$. Given $g\in H^{-1/2}(\Gamma)$, find $u\in H^1(\Omega)$ such that
\begin{equation}\label{DtN_Continuous_P}
\left\{
\begin{array}{rlll}
\Delta u + k^2 u &=& 0,&\rm{in}\ \Omega,\\
\frac{\partial u}{\partial \nu}&=&g,&\rm{on}\ \Gamma,\\
\frac{\partial u}{\partial \nu}&=&T(k)u,&\rm{on}\ \Gamma_R.
\end{array}
\right.
\end{equation}

The variational formulation for (\ref{DtN_Continuous_P}) is to find $u\in H^1(\Omega)$ such that
\begin{equation}\label{DtN-NEP}
(\nabla u, \nabla v) - k^2(u,v) - \langle T(k)u,v\rangle_{\Gamma_R} = \langle g,v\rangle,\quad \forall v\in H^1(\Omega),
\end{equation}
where $(\cdot,\cdot)$ denotes the $L^2(\Omega)$ inner product,
$\langle \cdot,\cdot\rangle$ the $H^{-1/2}(\Gamma)$-$H^{1/2}(\Gamma)$ duality,
and $\langle \cdot,\cdot\rangle_{\Gamma_R}$ the 
$H^{-1/2}(\Gamma_R)$-$H^{1/2}(\Gamma_R)$ duality.

Let $(\cdot,\cdot)_1$ be the $H^1(\Omega)$ inner product. Define the operators $B(k) : H^1(\Omega)\rightarrow H^1(\Omega)$, $k\in \Lambda$, and $M : H^{-1/2}(\Gamma)\rightarrow H^1(\Omega)$, respectively, by
\begin{equation}\label{B}
\left(B(k)u,v\right)_1 = (\nabla u, \nabla v) - k^2(u,v) - \langle T(k)u,v\rangle_{\Gamma_R},\quad \forall v\in H^1(\Omega)
\end{equation}
and
\begin{equation}\label{M}
\left(M g,v\right)_1 = \langle g,v\rangle, \quad \forall v\in H^1(\Omega).
\end{equation}
The operator form of \eqref{DtN-NEP} reads
\begin{equation}\nonumber
B(k)u = Mg,
\end{equation}
which defines the solution operator $B(k)^{-1}M$ that maps $g$ to $u$ for ${\rm Im}(k) \ge 0$. 

The scattering resonances are the poles of $B(\cdot)^{-1}$. There is a one-to-one match between the poles of $B(\cdot)^{-1}$ and the eigenvalues of $B(\cdot)$ (see \cite{Lenoir1992} or Theorem C.10 of \cite{SemyonMaciej}). The eigenvalue problem is to find $(\lambda,u)\in\Lambda\times H^1(\Omega)$ such that
\begin{equation}\label{DtN-operator-fun-EVP}
B(\lambda)u=0, \ \ \ {\rm in}\ H^1(\Omega).
\end{equation}

We shall show that $B(\cdot)$ is a holomorphic Fredholm operator function, for which some spectral properties are recalled as follows. Let $X, Y$ be complex Banach spaces and $\Lambda\subseteq \mathbb C$. Denote by $\mathcal L(X, Y )$ the space of bounded linear operators from $X$ to $Y$. An operator $F\in\mathcal L(X,Y)$ is said to be Fredholm if (1) the subspace $\mathcal R(F)$ (range of $F$) is closed in $Y$; and (2) the subspace $\mathcal N(F)$ (null space of $F$) and the quotient space $Y/\mathcal R(F)$ are finite-dimensional. The index of a Fredholm operator $F$ is defined as
\begin{equation}\nonumber
{\rm ind}(F)=\mbox{dim} (\mathcal N(F))-\mbox{dim} (Y/\mathcal R(F)).
\end{equation}
For an operator function $F(\cdot): \Lambda\rightarrow \mathcal{L}(X,Y)$, the resolvent set $\rho(F)$ 
and the spectrum $\sigma(F)$ are defined respectively as
\[
 \rho(F)=\{ k \in \Lambda :F(k)^{-1} \mbox{ exists and is bounded}\}
 \]
 and
 \[
 \sigma(F)=\Lambda\backslash \rho(F).
\]

Denote by $\Phi_0(\Lambda,\mathcal{L}(X,Y))$ the set of all holomorphic operator functions $F(\cdot)$'s such that $F(k) \in \mathcal{L}(X,Y)$ is Fredholm with index zero for each $k \in \Lambda$. For $F(\cdot)\in \Phi_0(\Lambda,\mathcal{L}(X,Y))$, the eigenvalue problem is to find $(\lambda,x)\in \Lambda\times X, \ x\neq 0$, such that
\begin{equation}
F(\lambda)x=0, \quad \text{in } Y.
\end{equation}

\begin{definition} An ordered sequence of elements $x_0, x_1, \dots, x_j$ in $X$ is called a Jordan chain of $F$ at an eigenvalue $\lambda$ if
\begin{eqnarray}
F(\lambda)x_j+\frac{1}{1!} F^{(1)}(\lambda)x_{j-1}+\cdots+\frac{1}{j!} F^{(j)}(\lambda)x_{0}=0,\quad j=0,1,\cdots,
\end{eqnarray}
where $F^{(j)}$ denotes the $j$-th derivative.
\end{definition}

The length of any Jordan chain of an eigenvalue is finite. 
Elements of any Jordan chain of an eigenvalue $\lambda$ are called generalized eigenelements of $\lambda$.

\begin{definition}
The closed linear hull of all generalized eigenelements of an eigenvalue $\lambda$, denoted by $G(\lambda)$, is called the generalized eigenspace of $\lambda$.
\end{definition}

If $F(\cdot)\in \Phi_0(\Lambda,\mathcal{L}(X,Y))$ and its resolvent set is not empty, then
the spectrum $\sigma(F)$ has no cluster points in $\Lambda$, and every $\lambda \in \sigma(F)$ is an eigenvalue with finite dimensional generalized eigenspace $G(\lambda)$. 

Now we show that $B(\cdot)$ is a holomorphic Fredholm operator function
with index 0. We first split the operator $B(k)$ defined in \eqref{B} (see Theorem 4.2 in \cite{Hsiao2011}). Using the recurrence relation
\begin{equation}\label{recur}
    \frac{{H^{(1)}_n}^{\prime}(kR)}{H^{(1)}_n(kR)} = 
    \frac{H^{(1)}_{n-1}(kR)}{H^{(1)}_n(kR)} - \frac{n}{kR},
\end{equation}
it holds that
\begin{align}\nonumber
    -\langle T(k)u,v\rangle_{\Gamma_R} &= -\sum_{n=0}^{+\infty}{\!}^{{}^{\,\scriptstyle{\prime}}}k\pi\frac{{H^{(1)}_n}^{\prime}(kR)}{H^{(1)}_n(kR)}
    (a^u_n\overline{\displaystyle a^v_n} 
    + b^u_n\overline{\displaystyle b^v_n})R
    \\
    \label{Tu,v}
    &=\pi\sum_{n=0}^{+\infty}{\!}^{{}^{\,\scriptstyle{\prime}}}
    \left(n - kR\frac{H^{(1)}_{n-1}(kR)}{H^{(1)}_n(kR)}\right)(a^u_n\overline{\displaystyle a^v_n} 
    + b^u_n\overline{\displaystyle b^v_n}),
\end{align}
where
\begin{equation}\nonumber
    a^u_n = \frac{1}{\pi}\int_0^{2\pi}u(R,\phi)\cos(n\phi)\text{d}\phi,
    \quad
    b^u_n = \frac{1}{\pi}\int_0^{2\pi}u(R,\phi)\sin(n\phi)\text{d}\phi.
\end{equation}
Let $\epsilon$ be a fixed number such that\! $0<\epsilon<1/2$.
Define two operators $A: H^1(\Omega)\rightarrow H^1(\Omega)$
and $K(k): H^1(\Omega)\rightarrow H^1(\Omega)$, $k\in \Lambda$, by
\begin{align}\nonumber
    (Au,v)_1 &= (\nabla u,\nabla v) + (u,v) + \pi\sum_{n=0}^{+\infty}{\!}^{{}^{\,\scriptstyle{\prime}}}
n
(a_n^u\overline{\displaystyle a_n^v} 
+ b_n^u\overline{\displaystyle b_n^v}),
\\
\nonumber
(K(k)u,v)_1 &= ((k^2+1)u,v) + \pi\sum_{n=0}^{+\infty}{\!}^{{}^{\,\scriptstyle{\prime}}}
kR\frac{H^{(1)}_{n-1}(kR)}{H^{(1)}_n(kR)}
(a_n^u\overline{\displaystyle a_n^v} 
+ b_n^u\overline{\displaystyle b_n^v}),
\end{align}
for all $v\in H^1(\Omega)$.
Make use of the norm on $H^s(\Gamma_R)$,
\begin{equation*}
\Vert u\Vert_{H^s(\Gamma_R)}^2 = \sum_{n=0}^{+\infty}{\!}^{{}^{\,\scriptstyle{\prime}}}(n^2+1)^s(|a_n^u|^2 + |b_n^u|^2),
\end{equation*}
and Lemma 3.1 of \cite{Hsiao2011}, which states that there exists
a constant $C$ dependent on $kR$ such that
\begin{equation*}
\left|\frac{H^{(1)}_{n-1}(kR)}{H^{(1)}_n(kR)}\right|\leqslant C,\quad \forall n\in\mathbb{N}.
\end{equation*}
One has that
\begin{align}\nonumber
\nonumber
\left|\pi\sum_{n=0}^{+\infty}{\!}^{{}^{\,\scriptstyle{\prime}}}
n
(a_n^u\overline{\displaystyle a_n^v} 
+ b_n^u\overline{\displaystyle b_n^v})\right|
&\leq C\Vert u\Vert_{H^{1/2}(\Gamma_R)}\Vert v\Vert_{H^{1/2}(\Gamma_R)}\leq C\Vert u\Vert_1\Vert v\Vert_1,
\\
\nonumber
\left|\pi\sum_{n=0}^{+\infty}{\!}^{{}^{\,\scriptstyle{\prime}}}
kR\frac{H^{(1)}_{n-1}(kR)}{H^{(1)}_n(kR)}
(a_n^u\overline{\displaystyle a_n^v} 
+ b_n^u\overline{\displaystyle b_n^v})\right|
&\leq C\Vert u\Vert_{L^2(\Gamma_R)}\Vert v\Vert_{L^2(\Gamma_R)}
\leq C\Vert u\Vert_{1/2+\epsilon}\Vert v\Vert_1,
\end{align}
where $C$ is a generic constant,
and $\Vert\cdot\Vert_s$
denotes the $H^s(\Omega)$-norm.
Therefore, the operators $A$ and $K(k)$ are well-defined,
and the extension of $K(k)$ on $H^{1/2+\epsilon}(\Omega)$
is bounded. We have that
\begin{equation}\label{B split}
B(k) = A - K(k).
\end{equation}
\begin{lemma}\label{holomorphicB}
    The operator function $B(\cdot)$ is holomorphic Fredholm with index $0$ on $\Lambda$. Furthermore, the resolvent set $\rho(B)$ is non-empty.
\end{lemma}
\begin{proof}
    For $k\in \Lambda$, since $T(k)$ is bounded,
    $B(k)$ is bounded.
    By the recurrence formula \eqref{recur}, it holds that
    \begin{align}\nonumber
        \left(k\frac{H^{(1)}_{n-1}(kR)}{H^{(1)}_n(kR)}\right)^{\prime} 
        &= \frac{H^{(1)}_{n-1}(kR)}{H^{(1)}_n(kR)}
        +kR\frac{{H^{(1)}_{n-1}}^{\prime}(kR)}{H^{(1)}_n(kR)}
        -kR\frac{H^{(1)}_{n-1}(kR)}{H^{(1)}_n(kR)}\frac{{H^{(1)}_n}^{\prime}(kR)}{H^{(1)}_n(kR)}
        \\
        \nonumber
        &= \frac{H^{(1)}_{n-1}(kR)}{H^{(1)}_n(kR)}\left(
        1+kR\frac{{H^{(1)}_{n-1}}^{\prime}(kR)}{H^{(1)}_{n-1}(kR)}
        -kR\frac{{H^{(1)}_n}^{\prime}(kR)}{H^{(1)}_n(kR)}
        \right)
        \\
        \nonumber
        &= \frac{H^{(1)}_{n-1}(kR)}{H^{(1)}_n(kR)}\left(
        2+kR\frac{H^{(1)}_{n-2}(kR)}{H^{(1)}_{n-1}(kR)}
        -kR\frac{H^{(1)}_{n-1}(kR)}{H^{(1)}_n(kR)}
        \right).
    \end{align}
    For $u, v\in H^1(\Omega)$, we take the derivative of \eqref{Tu,v} to obtain
    \begin{equation}\nonumber
        |\langle T^{\prime}(k)u,v\rangle_{\Gamma_R}|\leq
        C\sum_{n=0}^{+\infty}{\!}^{{}^{\,\scriptstyle{\prime}}}
        (n^2+1)^{1/2}(a^u_n\overline{\displaystyle a^v_n} 
        + b^u_n\overline{\displaystyle b^v_n})
        \leq C\Vert u\Vert_1\Vert v\Vert_1,
    \end{equation}
    which implies that $T^{\prime}(k):H^1(\Omega)\rightarrow H^1(\Omega)$
    is well-defined and bounded.
    Taking the derivative of \eqref{B}, we see that $B(\cdot)$ is holomorphic
    with $B^{\prime}(k):H^1(\Omega)\rightarrow H^1(\Omega)$ satisfying
    \begin{equation}\nonumber
        (B^{\prime}(k)u,v)_1 = -2k(u,v) 
        - \langle T^{\prime}(k)u,v\rangle_{\Gamma_R},\quad 
        \forall v\in H^1(\Omega).
    \end{equation}
    Since 
    \[
        (Au,u)_1 = (\nabla u,\nabla u) + (u,u) + \pi\sum_{n=0}^{+\infty}{\!}^{{}^{\,\scriptstyle{\prime}}}
        n(|a_n^u|^2 + |b_n^u|^2) \geq \Vert u\Vert_1^2, \quad \forall u\in H^1(\Omega),
    \]
    $A$ is invertible.  For $k\in \Lambda$,
    since the extension of $K(k)$ on $H^{1/2+\epsilon}(\Omega)$ is bounded and $H^1(\Omega)$ is compactly imbedded in $H^{1/2+\epsilon}(\Omega)$, $K(k)$ is compact. Consequently, $B(k)$ is a Fredholm operator of index $0$ for $k \in \Lambda$. 
    
    Since $B(k)$ is invertible for $k$ in the upper half-plane, its resolvent set is non-empty.
\end{proof}

For practical computation, one usually truncates the DtN mapping \eqref{DtN-mapping} to obtain $T^N(k): H^{1/2}(\Gamma_R)\rightarrow H^{-1/2}(\Gamma_R)$ such that
\begin{equation}\label{DtN-discrete-mapping}
T^N(k)\varphi = \sum_{n=0}^{N}{\!}^{{}^{\,\scriptstyle{\prime}}}\frac{k}{\pi}\frac{{H^{(1)}_n}^{\prime}(kR)}{H^{(1)}_n(kR)}\int_0^{2\pi}\varphi(\phi)\cos(n(\theta-\phi))\text{d}\phi.
\end{equation}
The integer $N$ is called the truncation order. Similar to \eqref{Tk}, $T^N(k)$ is bounded, i.e.,
\begin{equation}\label{TNk}
    \Vert T^N(k)\varphi\Vert_{H^{-1/2}(\Gamma_R)}\leq C\Vert \varphi\Vert_{H^{1/2}(\Gamma_R)}.
\end{equation}

The following lemma provides the error estimate for $T^N(k)$.
\begin{lemma}\label{lemma24} (Theorem 3.3 of \cite{Hsiao2011})
For the DtN mapping $T(k)$ and the truncated DtN mapping $T^N(k)$ defined, respectively, by (\ref{DtN-mapping}) and (\ref{DtN-discrete-mapping}), it holds that
\begin{equation}\nonumber
\|(T(k)-T^N(k))u\|_{H^{s-1}(\Gamma_R)}\leq C \frac{\gamma_{s+t}(N,u)}{N^t}\|u\|_{H^{s+t}(\Gamma_R)},\ \forall u\in H^{s+t}(\Gamma_R),\ s\geq\frac{1}{2},\ t\geq 0,
\end{equation}
where
\begin{equation}\label{55555}
\gamma_{s+t}(N,u)=
\frac{   \Big\{  \sum_{n=N+1}^{\infty}(1+n^2)^{s+t}(|a_n^u|^2+|b_n^u|^2)  \Big\}^{\frac{1}{2}}   }
     {   \left\{ \frac{|a_0^u|^2}{2}+\sum_{n=1}^{\infty}(1+n^2)^{s+t}(|a_n^u|^2+|b_n^u|^2)\right\}^{\frac{1}{2}} }\in [0,1]
\end{equation}
and $C>0$ is a constant depending on $kR$ but not on $u$ and $N$ .
\end{lemma}
Using $T^N(k)$, the approximation $B^N(k)$ of $B(k)$ is defined by
\begin{equation}\label{BN}
    (B^N(k)u,v)_1 = (\nabla u, \nabla v) - k^2(u,v) - \langle T^N(k)u,v\rangle_{\Gamma_R},\quad \forall v\in H^1(\Omega).
\end{equation}
The associated operator form is
\begin{equation}\label{Truncated-EVP}
    B^N(k)u = Mg.
\end{equation}
Similar to \eqref{B split}, we write
\begin{equation}\label{BN split}
B^N(k) = A^N - K^N(k),
\end{equation}
where $A^N : H^1(\Omega)\rightarrow H^1(\Omega)$ 
and $K^N(k) : H^1(\Omega)\rightarrow H^1(\Omega)$ are defined, respectively, by 
\[
(A^Nu,v)_1 = (\nabla u,\nabla v) + (u,v) + \pi\sum_{n=0}^{N}{\!}^{{}^{\,\scriptstyle{\prime}}}
n
(a_n^u\overline{\displaystyle a_n^v} 
+ b_n^u\overline{\displaystyle b_n^v}),
\]
and
\[
(K^N(k)u,v)_1 = ((k^2+1)u,v) + \pi\sum_{n=0}^{N}{\!}^{{}^{\,\scriptstyle{\prime}}}
kR\frac{H^{(1)}_{n-1}(kR)}{H^{(1)}_n(kR)}
(a_n^u\overline{\displaystyle a_n^v} 
+ b_n^u\overline{\displaystyle b_n^v}).
\]
Note that the extension of $K^N(k)$ on $H^{1/2+\epsilon}(\Omega)$ is also bounded.

The eigenvalue problem using $T^N(k)$ is to find $(\lambda,u)\in\Lambda\times H^1(\Omega)$ such that
\begin{equation}\label{DtN-operator-truncated-fun-EVP}
B^N(\lambda)u=0, \ \ \ {\rm in}\ H^1(\Omega).
\end{equation}

% ---------------------------------------------------------------------
\section{Finite Element Approximation}\label{FE}
In this section, we employ the linear Lagrange element method to discretize \eqref{DtN-operator-truncated-fun-EVP} and prove the convergence of eigenvalues. In particular, we use the abstract approximation theory for holomorphic Fredholm operator functions, which relies on the notion of regular convergence. We first introduce the definitions of operator convergence and present the approximation theory (see \cite{karma1996a,karma1996b,Vainikko1976} for more details). We prove two sufficient conditions for regular convergence. Then we show that the finite element approximation is regular and thus the convergence of eigenvalues holds.

Let $X, Y$ be Banach spaces.
Let $X_n, Y_n, n\in \mathbb{N},$ be
the respective sequences of approximation Banach spaces for $X$ and $Y$. Denote by $\mathcal{P}$ and $\mathcal{Q}$ two sequences of operators $\mathcal{P}=\{p_n\}_{n\in \mathbb{N}}$, $\mathcal{Q}=\{q_n\}_{n\in \mathbb{N}}$ with $p_n\in \mathcal{L}(X,X_n)$ and $q_n\in \mathcal{L}(Y,Y_n)$ such that
\begin{equation}
\lim_{n\to \infty} \| p_n x \|_{X_n}=\| x\|_X,\ \forall x\in X, \quad \lim_{n\to \infty} \| q_n y \|_{Y_n}=\| y\|_Y,\ \forall y\in Y.
\end{equation}

\begin{definition}
    A sequence $\{y_n\}_{n\in \mathbb{N}}$ with $y_n\in Y_n$
is $\mathcal{Q}$-compact if,
for each subsequence $\{y_n\}_{n\in \mathbb{N}^{\prime}}$, 
$\mathbb{N}^{\prime}\subseteq \mathbb{N}$,
there exists $\mathbb{N}^{\prime\prime}\subseteq \mathbb{N}^{\prime}$
and $y\in Y$ such that \! $\Vert y_n - q_n y\Vert_{Y_n}\rightarrow 0$,
$n\rightarrow \infty$, $n\in \mathbb{N}^{\prime\prime}$.
A sequence $\{x_n\}_{n\in \mathbb{N}}$ with $x_n\in X_n$
being $\mathcal{P}$-compact is defined similarly.
\end{definition}

For an operator $F\in \mathcal{L}(X,Y)$ and a sequence of operators $F_n\in \mathcal{L}(X_n,Y_n)$, we define the convergence, uniform convergence, compact convergence, stable convergence, and regular convergence as follows.
\begin{itemize}
\item    $F_n$ converges to $F$
    if 
    \begin{equation}\nonumber
    \Vert F_n p_n x - q_n F x\Vert_{Y_n} \rightarrow 0, 
    \quad n\rightarrow \infty, \quad\forall x\in X.
    \end{equation}

\item    $F_n$ converges uniformly to
    $F$ if 
    \begin{equation}\nonumber
        \Vert F_n p_n - q_n F\Vert \rightarrow 0,\quad n\rightarrow \infty.
    \end{equation}

\item    $F_n$ converges compactly to 
    $F$ if $F_n$ converges to $F$, and $\{F_n x_n\}$ is $\mathcal{Q}$-compact for any $\{x_n\}$ with $\Vert x_n\Vert_{X_n}\leq 1, \forall n\in \mathbb{N}$.    

\item    $F_n$ converges stably to
    $F$ if $F_n$ converges to $F$,
    and there exists $n_0$ such that $F_n$ is invertible
    with $\Vert F_n^{-1}\Vert\leq C$, $\forall n\geq n_0$.

\item    $F_n$ converges regularly to
    $F$ if $F_n$ converges to $F$,
    and $\{x_n\}$ is $\mathcal{P}$-compact
    for any $\{x_n\}$ such that $\{F_n x_n\}$ is $\mathcal{Q}$-compact.
\end{itemize}

Let $F(\cdot)\in \Phi_0(\Lambda,\mathcal L(X, Y))$ and $\rho(F)\neq \emptyset$.
To approximate the eigenvalues of $F(\cdot)$, we consider a sequence of operator functions $F_n(\cdot) \in \Phi_0(\Lambda,\mathcal L(X_n, Y_n )), \ n \in \mathbb N$. 
If $F_n(\cdot)$ and $F(\cdot)$ satisfy the following properties:
\begin{itemize}

\item[1.] $\{F_n(\cdot)\}_ {n\in \mathbb N}$ is equibounded on any 
compact $\Lambda_0\subseteq \Lambda$, i.e., 
\begin{equation}
\| F_n(k) \|\leq C, \quad \forall k\in \Lambda_0,\ n\in \mathbb N,
\end{equation}

\item[2.]  $F_n(k)$ converges regularly to
$F(k)$ for every $k\in \Lambda$,
\end{itemize}
then the following theorem on the convergence of eigenvalues of $F_n(\cdot)$ holds \cite{karma1996b}.

\begin{theorem}\label{KarmaThm} 
    Let $\Lambda_0$ be a compact subset of $\Lambda$ such that 
    $\partial \Lambda_0\subseteq \rho(F)$ and $\sigma(F)\cap \Lambda_0 = \{\lambda\}$.
    Then for sufficiently large $n$,
    there exist $\lambda_n\in \sigma(F_n)\cap \Lambda_0$  such that
    \begin{equation*}
        |\lambda_n - \lambda| \leq C \epsilon_n^{1/r},
    \end{equation*}
    where
    \begin{equation*}
        \epsilon_n = \sup_{\substack{k\in \partial \Lambda_0\\ u\in G(\lambda), \Vert u\Vert_1 = 1}}\Vert F_n(k)p_n u - q_n F(k)u\Vert_{Y_n},
    \end{equation*}
    and $r$ is the maximum length of Jordan chain of $\lambda$,
    and $G(\lambda)$ is the generalized eigenspace associated to $\lambda$. Moreover,
    the multiplicity of $\lambda$ and $\lambda_n$ coincides.
\end{theorem}

The regular convergence plays a key role for the convergence of eigenvalues. Let $E, H\in \mathcal{L}(X,Y)$ and $E_n, H_n\in \mathcal{L}(X_n,Y_n)$. A sufficient condition for the regular convergence of $E_n+H_n$ to $E+H$ is stated in the following lemma (see Theorem 2.55 of \cite{Vainikko1976}).

\begin{lemma}\label{regular1}
    Assume that $E$ is invertible, $H$ is compact, $E_n\in \mathcal{L}(X_n,Y_n)$ converges stably to
    $E\in \mathcal{L}(X,Y)$, and
    $H_n\in \mathcal{L}(X_n,Y_n)$ converges compactly to
    $H\in \mathcal{L}(X,Y)$.
    Then $E_n+H_n$ converges regularly to $E+H$.
\end{lemma}

Using Lemma \ref{regular1}, we show two more sufficient conditions for the regular convergence, which will be used to prove the regular convergence of $B^N_n(k)$ to $B(k)$. 

\begin{lemma}\label{lem1}
    Assume that $X_n \subseteq X$, $Y_n \subseteq Y$, $p_n$ and $q_n$ are projections such that
 $\Vert x - p_n x\Vert_X \rightarrow 0$, 
 $\Vert y - q_n y\Vert_Y\rightarrow 0$ as $n\rightarrow \infty$
for all $x\in X$ and $y\in Y$,
$E$ is invertible, $E_n$ are invertible with 
$\Vert E_n^{-1}\Vert\leq C$ for all $n$,
$E_n$ converges to $E$, $H$ is compact, $H_n$ converges uniformly to $H$.
Then $E_n + H_n$ converges regularly to $E + H$.
\end{lemma}

\begin{proof}
    By the invertibility of $E_n$, 
$E_n$ converges stably to $E$.
In addition, given $\Vert x_n\Vert_{X_n}\leq 1$, 
by the compactness of $H$, 
for each $\mathbb{N}^{\prime}\subseteq \mathbb{N}$, 
there exists $\mathbb{N}^{\prime\prime}\subseteq \mathbb{N}^{\prime}$ 
and $y\in Y$
such that 
\begin{equation*}
\Vert H x_n - y\Vert_Y \rightarrow 0, 
\quad n\rightarrow \infty, \quad n\in \mathbb{N}^{\prime\prime}.
\end{equation*}
Therefore,
\begin{align*}
\Vert H_n x_n - q_n y\Vert_{Y_n} &\leq \Vert (H_n p_n - q_n H)x_n\Vert_{Y_n} + \Vert q_n(H x_n - y)\Vert_{Y_n} 
\\
&\leq \Vert H_n p_n - q_n H\Vert\, \Vert x_n\Vert_{X_n} + \Vert q_n\Vert\, \Vert H x_n - y\Vert_Y \rightarrow 0,
\quad n\rightarrow \infty, \quad n\in \mathbb{N}^{\prime\prime},
\end{align*}
which implies that $H_n$ converges compactly to $H$. Therefore, $E_n+H_n$ converges regularly to $E+H$ by Lemma~\ref{regular1}.
\end{proof}

\begin{lemma}\label{lem2}
    Assume that $X_n \subseteq X$, $Y_n \subseteq Y$, $p_n$ and $q_n$ are projections such that
 $\Vert x - p_n x\Vert_X \rightarrow 0$, 
 $\Vert y - q_n y\Vert_Y\rightarrow 0$ as $n\rightarrow \infty$
for all $x\in X$ and $y\in Y$,
$E$ is invertible, $E_n$ are invertible with 
$\Vert E_n^{-1}\Vert\leq C$ for all $n$, 
$E_n$ converges to $E$,
$H$ is compact, and $H_n = q_n H|_{X_n}$. Then
$E_n + H_n$ converges regularly to $E + H$.
\end{lemma}

\begin{proof}
    $H_n$ converges to $H$ because
\begin{equation*}
\Vert H_n p_n x - q_n H x\Vert_{Y_n} = \Vert q_n H(p_n x - x)\Vert_{Y_n} \leq \Vert q_n\Vert \,\Vert H\Vert\, \Vert p_n x - x\Vert_X\rightarrow 0,\quad n\rightarrow \infty,\quad \forall x\in X.
\end{equation*}
Since
\begin{equation*}
\Vert H_n x_n - q_n y\Vert_{Y_n} = \Vert q_n(H x_n - y)\Vert_{Y_n} \leq \Vert q_n\Vert\, \Vert H x_n - y\Vert_Y \rightarrow 0,
\quad n\rightarrow \infty, \quad n\in \mathbb{N}^{\prime\prime},
\end{equation*}
the stable convergence of $E_n$ and the compact convergence of $H_n$
can be proved in the same way as the previous lemma.
\end{proof}

Now we introduce the finite element approximation. Let $\mathcal{T}_n:=\mathcal{T}_{h_n}$ be a regular triangular mesh for $\Omega$ with mesh size $h_n\rightarrow 0^+$ as $n\rightarrow\infty$. Let $V_n\subseteq H^1(\Omega)$ be the linear Lagrange finite element space on $\mathcal{T}_n$.  For $k\in\Lambda$, define $B^N_n(k) : V_n \rightarrow V_n$ such that
\begin{equation}\label{BnN}
    (B^N_n(k)u_n,v_n)_1 = (\nabla u_n, \nabla v_n) - k^2(u_n,v_n) - \langle T^N(k)u_n,v_n\rangle_{\Gamma_R},\quad \forall v_n\in V_n.
\end{equation}
In a similar way, we define $A^N_n : V_n\rightarrow V_n$ and $K^N_n(k):V_n\rightarrow V_n$, respectively, as the finite element approximation of $A^N$ and $K^N(k)$ such that 
\begin{equation}\nonumber
    B^N_n(k) = A^N_n - K^N_n(k).
\end{equation}

In the next a few lemmas, we prove the properties of the relevant operators.

\begin{lemma}\label{lemma1}
On each compact $\Lambda_0\subseteq \Lambda$,
$B^N(\cdot)$ and $B^N_n(\cdot)$ are equibounded, i.e.,
\begin{equation}\nonumber
    \Vert B^N(k)\Vert\leq C,\quad \Vert B^N_n(k)\Vert \leq C,
    \quad \forall k\in \Lambda_0, N\in \mathbb{N}, n\in \mathbb{N}.
\end{equation}
\end{lemma}
\begin{proof}
    By the boundedness \eqref{TNk}, the definitions \eqref{BN} and \eqref{BnN}, $B^N(k)$ and
    $B^N_n(k)$ are bounded.
\end{proof}

\begin{lemma}\label{invertible compact}
    For each $k\in\Lambda$, $K^N(k)$
    is compact. The operators $A^N$ and $A^N_n$ are invertible with
    \begin{equation*}
        \Vert (A^N)^{-1}\Vert\leq C,\quad
        \Vert (A^N_n)^{-1}\Vert\leq C,\quad
        \forall N\in\mathbb{N}, n\in \mathbb{N}.
    \end{equation*}
\end{lemma}

\begin{proof}
    The proof is similar to that of Lemma \ref{holomorphicB}.
\end{proof}

\begin{lemma}\label{holomorphic}
    For $N$ sufficiently large, $B^N(\cdot)$ is a holomorphic Fredholm operator function
    with index 0 on $\Lambda$,
    and its resolvent set is non-empty. 
    Moreover,
    $B^N_n(\cdot)$ is a holomorphic Fredholm operator function
    with index 0 on $\Lambda$ for all $N$ and $n$.
\end{lemma}

\begin{proof}
    Similar to the proof of Lemma~\ref{holomorphicB}, $B^N(\cdot)$ and $B^N_n(\cdot)$ are holomorphic operator functions on $\Lambda$. Since $A^N$ is invertible and $K^N(k)$ is compact, $B^N(k)$ is Fredholm. When $\text{Im}(k)>0$ and $N$ is sufficiently large, $B^N(k)$ is invertible \cite{Hsiao2011}. Therefore, $\rho(B^N)$ is non-empty. Since $V_n$ is finite dimensional,
    $B^N_n(k)$ is Fredholm of index $0$ for $k \in \Lambda$.
\end{proof}

We are now ready to prove the regular convergence of $B_n^N(k)$ to $B(k)$. We first show the convergence of $B^N(k)$ to $B(k)$ by taking $X^N = Y^N = X = Y = H^1(\Omega)$ and setting the projections $p^N$ and $q^N$ to be the identity operator from $H^1(\Omega)$ to itself.

\begin{lemma}\label{regular N}
    For $k\in \Lambda$, $B^N(k)$ converges regularly to $B(k)$.
\end{lemma}
\begin{proof}
    For $u\in H^1(\Omega)$, let $v = (A - A^N)u$. We have that
    \begin{align*}
        &\Vert q^N A u - A^N p^N u\Vert^2_1 = \Vert (A - A^N)u\Vert_1^2 = (Au,v)_1 - (A^Nu,v)_1
        \\
        =\;&\pi\sum_{n=N+1}^{+\infty}n
        (a_n^{u}\overline{\displaystyle a_n^{v}} + b_n^{u}\overline{\displaystyle b_n^{v}})
        \leq C\left(\sum_{n=N+1}^{+\infty}
        (n^2+1)^{\frac{1}{2}}(|a_n^{u}|^2 + |b_n^{u}|^2)\right)^{\!\!\!\frac{1}{2}} \Vert v\Vert_1.
    \end{align*}
    Therefore,
    \begin{equation*}
    \Vert q^N A u - A^N p^N u\Vert_1\leq C\gamma_{1/2}(N,u)\Vert u\Vert_1\rightarrow 0,\quad N\rightarrow \infty,\quad \forall u\in H^1(\Omega).
    \end{equation*}
    Thus $A^N$ converges to $A$. On the other hand,
    letting $v = (K(k) - K^N(k))u$, we have that
    \begin{align*}
    &\Vert q^N K(k) u - K^N(k) p^N u\Vert_1^2 = \Vert (K(k) - K^N(k))u\Vert_1^2 = (K(k)u,v)_1 - (K^N(k)u,v)_1
    \\
    =\;&\pi\sum_{n=N+1}^{+\infty}
    kR\frac{H^{(1)}_{n-1}(kR)}{H^{(1)}_n(kR)}
    (a_n^{u}\overline{\displaystyle a_n^{v}} + b_n^{u}\overline{\displaystyle b_n^{v}})
    \leq
    C\left(\sum_{n=N+1}^{+\infty}
    (|a_n^{u}|^2 + |b_n^{u}|^2)\right)^{\!\!\!\frac{1}{2}}\Vert v\Vert_{L^2(\Gamma_R)}
    \\
    \leq\;&
    \frac{C}{\sqrt{N}}\left(\sum_{n=N+1}^{+\infty}
    (n^2+1)^{\frac{1}{2}}(|a_n^{u}|^2 + |b_n^{u}|^2)\right)^{\!\!\!\frac{1}{2}}\Vert v\Vert_1
    \leq \frac{C}{\sqrt{N}}\Vert u\Vert_1\Vert v\Vert_1.
    \end{align*}
    Therefore,
    \begin{equation*}
    \Vert q^N K(k) u - K^{N}(k)p^N u\Vert_1 \leq \frac{C}{\sqrt{N}}\Vert u\Vert_1,\quad \forall u\in H^1(\Omega),
    \end{equation*}
    which implies that $K^N(k)$ converges uniformly to $K(k)$.
    By Lemma \ref{lem1}, $B^N(k)$ converges regularly to $B(k)$.
\end{proof}

For the convergence of $B^N_n(k)$ to $B^N(k)$,
we set $X = Y = H^1(\Omega)$, $X_n = Y_n = V_n$.
Let $q_n : H^1(\Omega)\rightarrow V_n$ be the $H^1$ projection 
\begin{equation}\nonumber
    (q_n u, v_n)_1 = (u,v_n)_1,\quad \forall v_n\in V_n,
\end{equation}
and $p_n : H^1(\Omega)\rightarrow V_n$ be defined by
\begin{equation}\nonumber
    (A^N p_n u, v_n)_1 = (A^N u, v_n)_1,\quad \forall v_n\in V_n.
\end{equation}
Since $(A^N u,v)_1$ is an inner product on $H^1(\Omega)$ and $V_n$ is a closed subspace of $H^1(\Omega)$, the projection $p_n$ is well-defined.

\begin{lemma}\label{regular n}
    For $k\in \Lambda$ and $N$ sufficiently large, 
    $B^N_n(k)$ converges regularly to $B^N(k)$.
\end{lemma}

\begin{proof}
    It is clear that
    \begin{equation*}
    \Vert u - q_n u\Vert_1 \rightarrow 0,\quad n\rightarrow \infty,\quad \forall u\in H^1(\Omega).
    \end{equation*}
    By the coercivity and boundedness of $A^N$, we have that
    \begin{equation*}
    \Vert u - p_n u\Vert_1^2 \leq (A^N(u - p_n u),u - p_n u)_1 = (A^N(u - p_n u),u - v_n)_1\leq C\Vert u - p_n u\Vert_1 \Vert u - v_n\Vert_1
    \end{equation*}
    for all $v_n \in V_n$. Hence
    \begin{equation*}
    \Vert u - p_n u\Vert_1 \leq C\inf_{v_n\in V_n}\Vert u - v_n\Vert_1\rightarrow 0,\quad
    n\rightarrow \infty,\quad \forall u\in H^1(\Omega).
    \end{equation*}
    It holds that
    \begin{equation*}
        (q_nA^Nu,v_n)_1 = (A^Nu,v_n)_1 = (A^Np_nu,v_n)_1,\quad
        \forall u\in H^1(\Omega), v_n\in V_n,
    \end{equation*}
    which implies $q_nA^N = A^Np_n$.
    Therefore, $A^N_n$ converges to $A^N$.
    We also have
    \begin{equation*}
        (q_nK^N(k)u_n,v_n)_1 = (K^N(k)u_n,v_n)_1 = (K^N_n(k)u_n,v_n)_1,
        \quad \forall u_n, v_n\in V_n.
    \end{equation*}
    Hence $q_n K^N(k)|_{V_n} = K_n^N(k)$.
    By Lemma \ref{lem2}, $B^N_n(k)$ converges regularly to $B^N(k)$.
\end{proof}

We are now ready to present the convergence of eigenvalues. 
\begin{theorem}
    Let $\Lambda_0$ be a compact subset of $\Lambda$ such that 
    $\partial \Lambda_0\subseteq \rho(B)$ and 
    $\sigma(B)\cap \Lambda_0 = \{\lambda\}$.
    Then for sufficiently large $N$ and $n$,
    there exist $\lambda^N\in \sigma(B^N)\cap \Lambda_0$ and 
    $\lambda^N_n\in \sigma(B^N_n)\cap \Lambda_0$ having the same
    multiplicity as $\lambda$, such that
    \begin{align*}
        |\lambda^N - \lambda|\leq C 
        \left(\sup_{\substack{u\in G(\lambda)\\\Vert u\Vert_1 = 1}}
        \gamma_{1/2}(N,u) + \frac{1}{\sqrt{N}}\right)^{1/r},
        \\
        |\lambda^N_n - \lambda^N|\leq C 
        \sup_{\substack{u\in G(\lambda^N)\\\Vert u\Vert_1 = 1}}
        \Vert u - p_n u\Vert_{1/2+\epsilon}^{1/r_N},
    \end{align*}
    where $r$ and $r_N$ are respectively the maximum length of Jordan chain of $\lambda$ and $\lambda^N$,
    $G(\lambda)$ and $G(\lambda^N)$ are the respective generalized eigenspace
    of $\lambda$ and $\lambda^N$ (finite-dimensional),
    and $C$ depends on $\Lambda_0$ but independent on $N$ or $n$.
\end{theorem}

\begin{proof}
    By Lemma \ref{regular N}, it holds that
    \begin{align*}
        \Vert q^N B(k)u - B^N(k)p^Nu\Vert_1
        &\leq 
        \Vert q^N Au - A^Np^Nu\Vert_1 + \Vert q^N K(k)u - K^N(k)p^N u\Vert_1
        \\
        &\leq C\gamma_{1/2}(N,u)\Vert u\Vert_1 
        + \frac{C}{\sqrt{N}}\Vert u\Vert_1.
    \end{align*}
    By Lemma \ref{regular n}, we have that
    \begin{align*}
        &\Vert q_n B^N(k)u - B^N_n(k)p_nu\Vert_1 
        = \Vert q_n K^N(k)u - K^N_n(k)p_nu\Vert_1
        \\
        \leq\;& \Vert q_n\Vert \Vert K^N(k)\Vert_
        {H^{1/2+\epsilon}(\Omega)\rightarrow H^{1}(\Omega)}
        \Vert u - p_nu\Vert_{1/2+\epsilon}
        \leq C\Vert u - p_n u\Vert_{1/2+\epsilon}.
    \end{align*}
    The proof is complete by applying Theorem \ref{KarmaThm}.
\end{proof}

\begin{remark}
If the functions in $G(\lambda)$ and $G(\lambda^N)$ belong to $H^{3/2-t}(\Omega)$ for some small $t>0$ and $\Vert u - p_n u\Vert_0\leq Ch_n\Vert u\Vert_1$ for all $u\in G(\lambda^N)$, then 
\begin{equation*}
        |\lambda^N - \lambda|\leq \frac{C}{N^{(1/2-t)/r}},
        \quad
        |\lambda^N_n - \lambda^N|\leq Ch_n^{(1-\epsilon-t)/r_N}.
\end{equation*}
\end{remark}

\begin{remark}
For each converging sequence of eigenvalues $\lambda^N_n$ of $B^N_n(\cdot)$, the limit is an eigenvalue of $B(\cdot)$ \cite{karma1996a}. In addition, the convergence of the eigenfunctions can also be obtained \cite{Beyn}.    
\end{remark}

%------------------------------------ Numerical Examples -------------------
\section{Numerical Examples}\label{NE}
%The DtN mapping and finite element discretization is combined to discretize (\ref{NeumannBCP}).
We present numerical results for three obstacles to demonstrate the effectiveness of the proposed finite element DtN method. Assume that the basis functions of $V_n$ are  $\phi_i,\ i=1,\cdots,N_n$. Let $S^1$, $S^2$, $S^3(k)$ be the matrices corresponding to
\begin{equation}\label{Continuous_matrix}\nonumber
(\nabla u_n,\nabla v_n),\; (u_n,v_n),\; \langle T^N(k)u_n,v_n\rangle_{\Gamma_R},
\end{equation}
respectively. The matrix version of the operator eigenvalue problem for (\ref{DtN-NEP}) is to find $k\in\mathbb{C}$ and a nontrivial $u_n\in V_n$ such that
\begin{equation}\label{matrix-form}
B_{n}^N(k)u_n:=(S^1-k^2S^2-S^3(k))u_n=0.
\end{equation}
The initial mesh $\mathcal{T}_{h_1}$ for $\Omega$ has the mesh size $h_1=\frac{\pi}{25}$. Then we uniformly refine $\mathcal{T}_{h_1}$ to obtain $\mathcal{T}_{h_2}$ etc. The relative error and the convergence order are denoted, respectively, by
$$
E_j=\frac{|k^{j}-k^{j+1}|}{|k^{j+1}|},\ \ \ j=1,2,3,4,
$$
and
$$
\log_2\left(\frac{E_{j}}{E_{j+1}}\right),\ \ \ j=1,2,3.
$$

The matrix eigenvalue problem \eqref{matrix-form} is nonlinear, which is an important topic in the numerical linear algebra community. Many methods such as iterative techniques and linearization algorithms have been investigated. Recently, the contour integral methods have become popular \cite{Asakura2009, Beyn2012,HuangSun,BrennanEtal2023}. Among them, the spectral indicator method (SIM) is very simple and easy to implement \cite{HuangSun,Gong2022MC}. It can be used as a standalone eigensolver or a screening tool for the distribution of eigenvalues. Two parallel spectral indicator methods were proposed in \cite{XiSun2023} to compute all the eigenvalues inside a given region on the complex plane. They do not require any prior knowledge of the number and distribution of the eigenvalues. The robustness are achieved by carefully choosing the parameters and adding a validation step. The parallel nature guarantees the high efficiency. In this paper, we employ the parallel SIM in \cite{XiSun2023} to compute all the eigenvalues of \eqref{matrix-form} in $S=[0,4]\times[-4,0]$ on the complex plane.

For the DtN mapping, we take a fixed truncation order $N=20$. The truncation order certainly affects the accuracy but $20$ is large enough for the scattering poles in $S$. The numerical experiments in \cite{Hsiao2011} show that the convergence rate of $T^N$ is very fast with respect to $N$ and the rate decreases when $kR$ increases.

\subsection{Unit Disk}%\label{soundsoft}
Let the obstacle $D$ be the unit disk centered at the origin. The exact scattering poles can be found analytically. The solution in polar coordinates to the scattering problem \eqref{NeumannBCP}-\eqref{SRC} can be written in the form
\begin{equation}\label{scattering}
u(r,\theta):=\sum_{m=-\infty}^\infty\alpha_m H_m^{(1)}(kr)e^{im\theta}.
\end{equation}
Assume that the incident wave $u^{\rm inc}$ is given by
\begin{equation}\label{incident}
u^{\rm inc}(r,\theta)=\sum_{m=-\infty}^\infty\beta_m J_m(kr)e^{im\theta}.
\end{equation}
Then $g(r,\theta) = -\frac{\partial u^{\rm inc}}{\partial \nu}$. Using \eqref{scattering} to match $\partial u(r, \theta)/\partial \nu$ and $g(r,\theta)$ on $\partial D$ ($r=1$), one has that
\begin{equation}
\alpha_m=-\frac{\beta_mJ'_m(k)}{H_m^{(1)'}(k)}.
\end{equation}
Hence the scattering poles for the sound hard unit disk are the zeros of $H_m^{(1)'}(k)$.

We take $R=1.25$ for $\Gamma_R$. The computed scattering poles on the finest mesh ($h_5$) are shown in Figure \ref{fig1}, which are consistent with the exact values.
\begin{figure}
\centering
\includegraphics[width=0.6\textwidth]{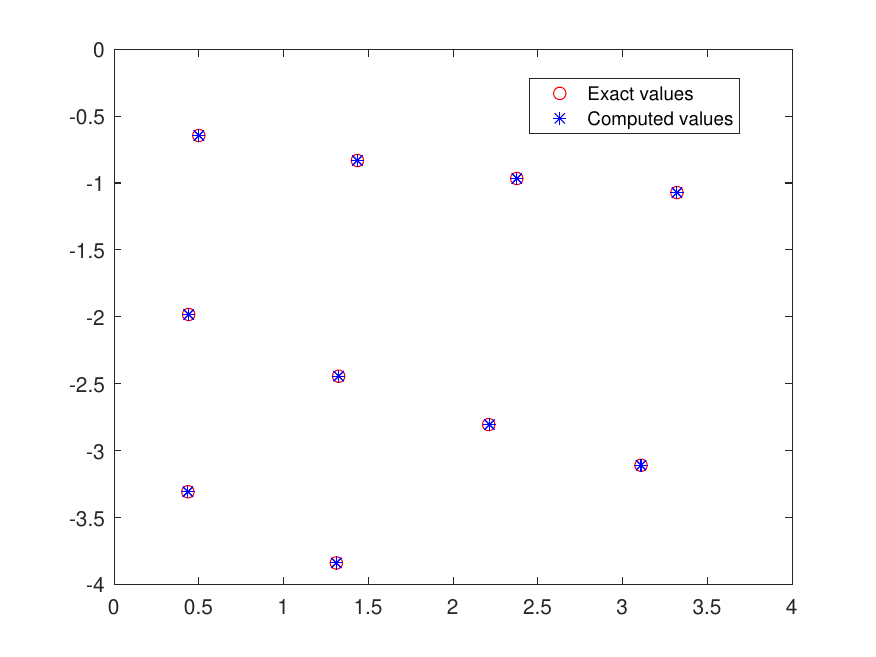}
\caption{The computed scattering poles and exact poles in $[0,4]\times[-4,0]$ for the unit disk.}
\label{fig1}
\end{figure}
The three small (in norm) scattering poles on different meshes are shown in Table \ref{table1}. Note that all the poles have approximately 2nd order of convergence. The associated eigenfunctions are shown in Figure \ref{fig-efs-1}.
\begin{table}
\begin{center}
\caption{Three small scattering poles and their convergence orders for the unit disk.}
\medskip
\label{table1}
\begin{tabular}{ccccccc}
\hline
%&$k_1$&Order& $k_2$&Order&$k_3$&Order\\
%\hline
%$h_1$   &0.5013-0.6441i&       &1.4357-0.8364i&       &0.4434-1.9897i&    \\
%$h_2$   &0.5012-0.6437i&       &1.4348-0.8350i&       &0.4415-1.9837i&    \\
%$h_3$   &0.5012-0.6436i&1.9896 &1.4345-0.8347i&2.0013 &0.4410-1.9821i&1.9915\\
%$h_4$   &0.5012-0.6436i&1.9956 &1.4344-0.8346i&2.0009 &0.4408-1.9817i&1.9924\\
%$h_5$   &0.5012-0.6435i&1.9985 &1.4344-0.8346i&2.0003 &0.4408-1.9817i&1.9968\\
 &$k_1$&Ord& $k_2$&Ord&$k_3$&Ord\\
\hline
$h_1$   &0.501344-0.644051i&       &1.435748-0.836362i&       &0.443441-1.989728i&    \\
$h_2$   &0.501225-0.643672i&       &1.434768-0.834997i&       &0.441463-1.983656i&    \\
$h_3$   &0.501194-0.643577i&2.0 &1.434521-0.834659i&2.0 &0.440967-1.982130i&2.0\\
$h_4$   &0.501186-0.643553i&2.0 &1.434459-0.834574i&2.0 &0.440842-1.981746i&2.0\\
$h_5$   &0.501184-0.643547i&2.0 &1.434443-0.834553i&2.0 &0.440810-1.981650i&2.0\\
\hline
\end{tabular}
\end{center}
\end{table}

\begin{figure}[ht]
\begin{center}
\begin{tabular}{ccc}
\resizebox{0.3\textwidth}{!}{\includegraphics{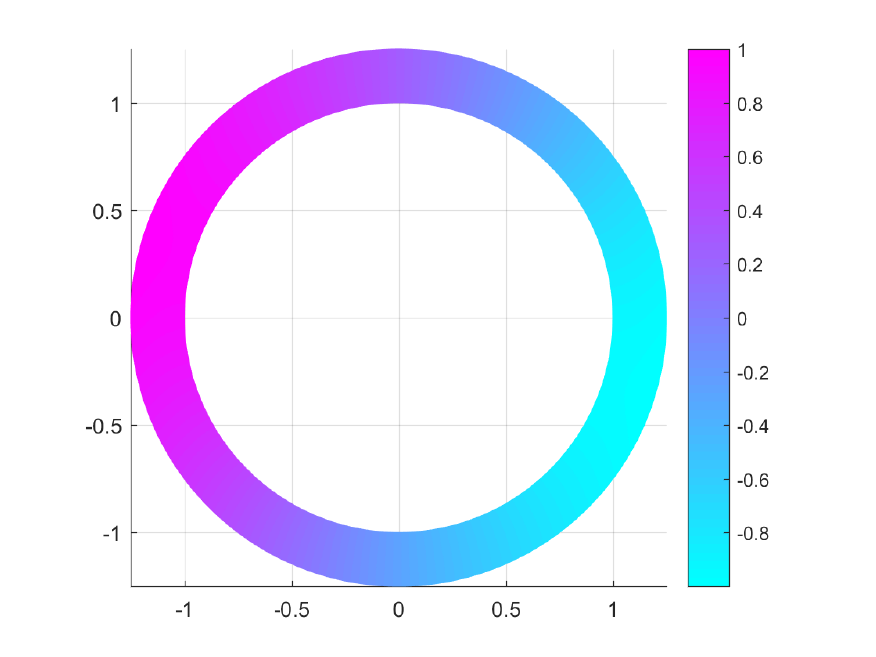}}&
\resizebox{0.3\textwidth}{!}{\includegraphics{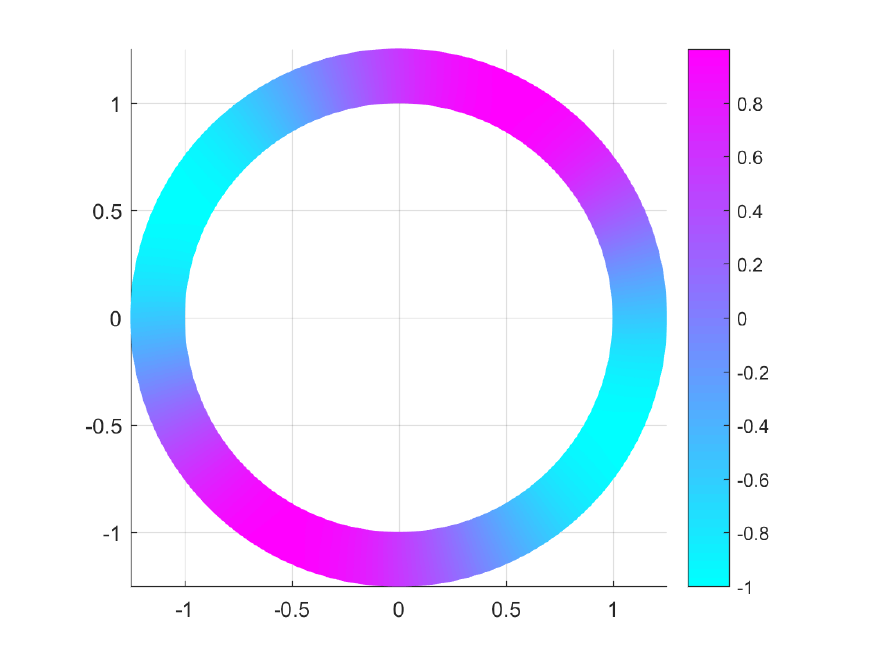}}&
\resizebox{0.3\textwidth}{!}{\includegraphics{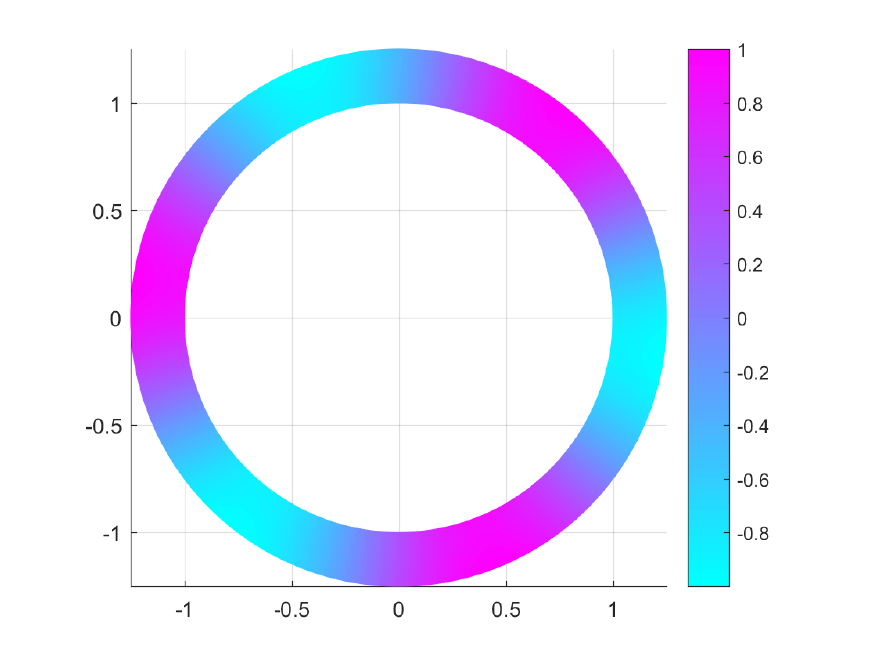}}
\end{tabular}
\end{center}
\caption{Real parts of the eigenfunctions associated with the three small scattering poles in {\bf Example 1}. }
\label{fig-efs-1}
\end{figure}

\subsection{Unit square}%\label{soundhard}
\begin{table}
\begin{center}
\caption{Three small scattering poles and their convergence orders for the unit square.}
\medskip
\label{table2}
\begin{tabular}{ccccccc}
\hline
 &$k_1$&Ord& $k_2$&Ord&$k_3$&Ord\\
\hline
%$h_1$   &0.880422-1.100729i&       &2.4388-1.0543i&    &2.4096-1.7587i&       \\
%$h_2$   &0.881359-1.097228i&       &2.4246-1.0347i&    &2.4127-1.7610i&       \\
%$h_3$   &0.881459-1.095541i&1.1 &2.4187-1.0270i&1.3179&2.4134-1.7616i&2.0810 \\
%$h_4$   &0.881434-1.094799i&1.2 &2.4162-1.0240i&1.3211&2.4135-1.7618i&2.0736 \\
%$h_5$   &0.881408-1.094487i&1.2 &2.4153-1.0228i&1.3256&2.4136-1.7618i&2.0562 \\
$h_1$   &0.880423-1.100515i&       &2.438825-1.054317i&    &2.409552-1.758703i&       \\
$h_2$   &0.881354-1.097152i&       &2.424593-1.034707i&    &2.412677-1.761034i&       \\
$h_3$   &0.881456-1.095512i&1.1 &2.418661-1.027046i&1.3&2.413375-1.761636i&2.1 \\
$h_4$   &0.881432-1.094788i&1.2 &2.416234-1.024029i&1.3&2.413536-1.761785i&2.1 \\
$h_5$   &0.881407-1.094483i&1.2 &2.415253-1.022837i&1.3&2.413574-1.761821i&2.1 \\
\hline
\end{tabular}
\end{center}
\end{table}

Let $D$ be the unit square given by $(-1/2,1/2)\times (-1/2,1/2)$ and $R=0.85$ for $\Gamma_R$. Eight scattering poles are found in $S$. The values on the finest mesh $\mathcal{T}_{h_5}$ are:
\[
\begin{array}{llll}
  0.881407 - 1.094483i,&
  0.881408 - 1.094491i,&
  2.415253 - 1.022837i,&
  2.413574 - 1.761821i,\\
  0.753148 - 3.384440i,&
  0.753157 - 3.384468i,&
  3.962260 - 1.585959i,&
  3.962286 - 1.585985i.
\end{array}
\]
Three small computed scattering poles are shown in Table \ref{table2}. The convergence order is less than~$2$. Note that $\Omega$ is the exterior to the square and has reentrant corners.
%which also shows that the proposed DtN method does obtain the theoretically
%predicted second convergence order.
The real parts of the eigenfunctions are shown in Figure \ref{fig-efs-2}.
\begin{figure}[ht]
\begin{center}
\begin{tabular}{ccc}
\resizebox{0.3\textwidth}{!}{\includegraphics{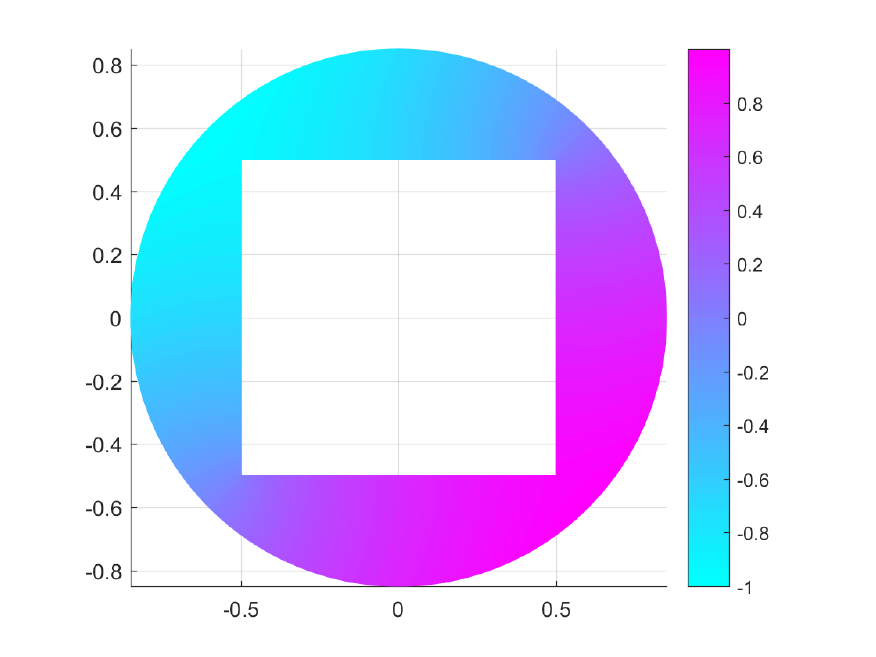}}&
\resizebox{0.3\textwidth}{!}{\includegraphics{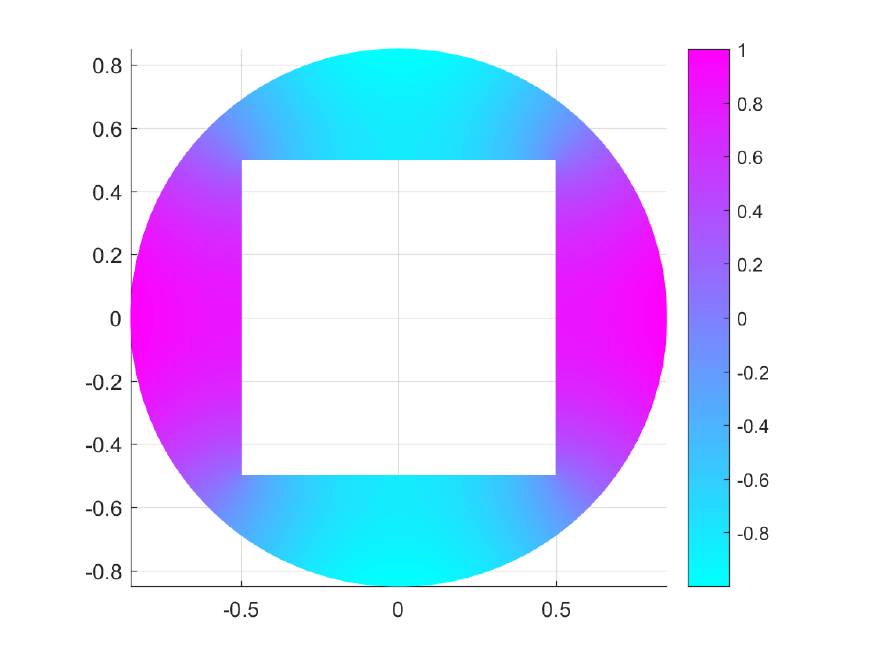}}&
\resizebox{0.3\textwidth}{!}{\includegraphics{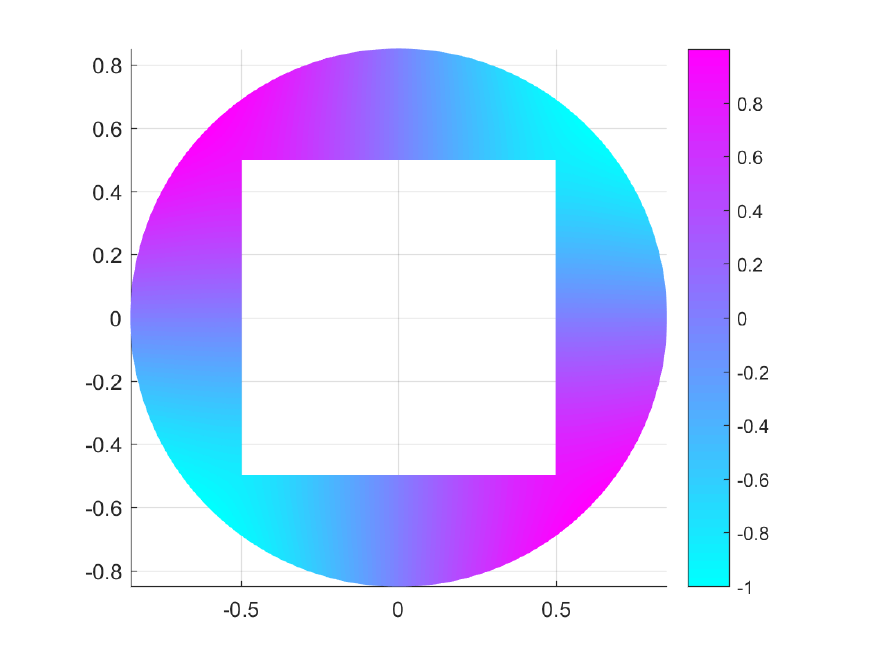}}
\end{tabular}
\end{center}
\caption{Real parts of the eigenfunctions associated with thethree small scattering poles in {\bf Example 2}. }
\label{fig-efs-2}
\end{figure}

\subsection{L-shaped Domain}
Let $D$ be the L-shapded domain given by $(-1/2,1/2)\times(-1/2,1/2)\backslash[0,1/2]\times[0,1/2]$. We take $R=0.85$ for $\Gamma_R$.  The computed poles in $S$ are as follows
\[
\begin{array}{llll}
1.061307-1.097954i,&0.916086-1.292350i,&2.594561-1.323668i,&2.663826-1.362413i,\\
0.824321-3.733939i,&0.873533-3.726105i,&3.896025-1.388263i.
\end{array}
\]
The convergence behavior of the first three computed scattering poles are presented in Table \ref{table3}. The real part of the eigenfunctions for the first three scattering poles are shown in Figure \ref{fig-efs-3}.
\begin{table}
\begin{center}
\caption{Three small scattering poles and their convergence orders for the L-shaped domain.}
\medskip
\label{table3}
\begin{tabular}{ccccccc}
\hline
% &$k_1$&Order& $k_2$&Order&$k_3$&Order\\
%\hline
%$h_1$   &1.0606-1.1074i&       &0.9144-1.3022i&    &2.6149-1.3639i&       \\
%$h_2$   &1.0616-1.1019i&       &0.9158-1.2965i&    &2.6030-1.3388i&       \\
%$h_3$   &1.0616-1.0994i&1.1874 &0.9161-1.2939i&1.1703&2.5977-1.3290i&1.3132 \\
%$h_4$   &1.0614-1.0984i&1.2234 &0.9161-1.2928i&1.2311&2.5955-1.3252i&1.3187 \\
%$h_5$   &1.0613-1.0980i&1.2623 &0.9161-1.2923i&1.2706&2.5946-1.3237i&1.3237 \\
 &$k_1$&Ord& $k_2$&Ord&$k_3$&Ord\\
\hline
$h_1$   &1.060612-1.107405i&       &0.914365-1.302247i&    &2.614854-1.363921i&       \\
$h_2$   &1.061650-1.101887i&       &0.915794-1.296528i&    &2.603012-1.338834i&       \\
$h_3$   &1.061557-1.099426i&1.2 &0.916057-1.293925i&1.2&2.597730-1.329038i&1.3 \\
$h_4$   &1.061399-1.098384i&1.2 &0.916090-1.292811i&1.2&2.595487-1.325187i&1.3 \\
$h_5$   &1.061307-1.097954i&1.3 &0.916086-1.292350i&1.3&2.594561-1.323668i&1.3 \\
\hline
\end{tabular}
\end{center}
\end{table}

\begin{figure}[ht]
\begin{center}
\begin{tabular}{ccc}
\resizebox{0.3\textwidth}{!}{\includegraphics{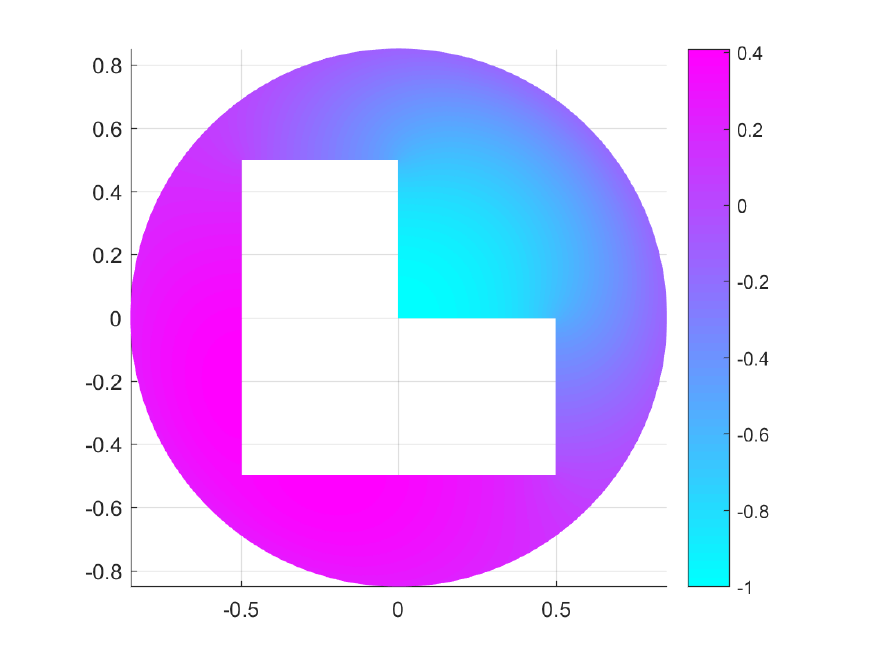}}&
\resizebox{0.3\textwidth}{!}{\includegraphics{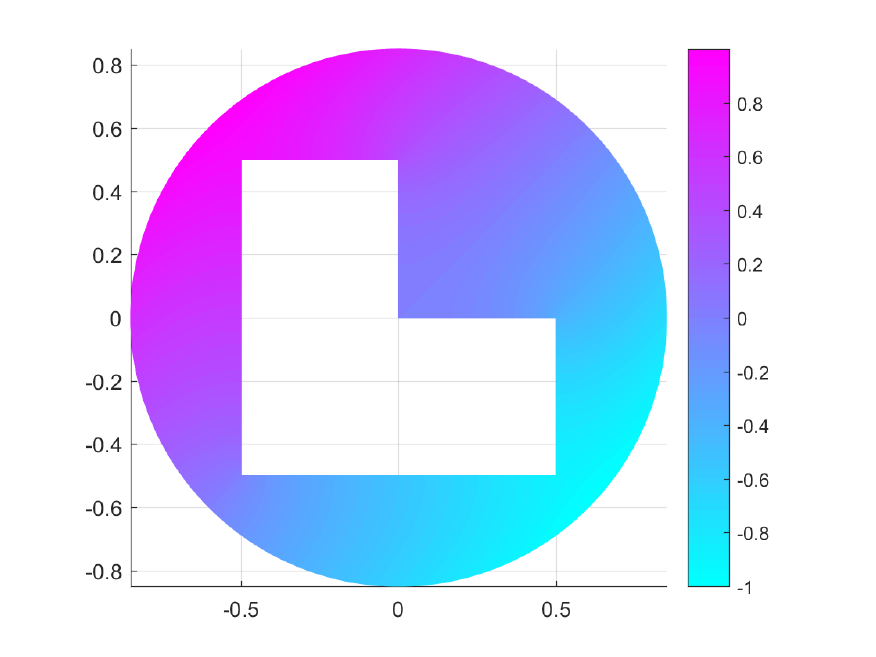}}&
\resizebox{0.3\textwidth}{!}{\includegraphics{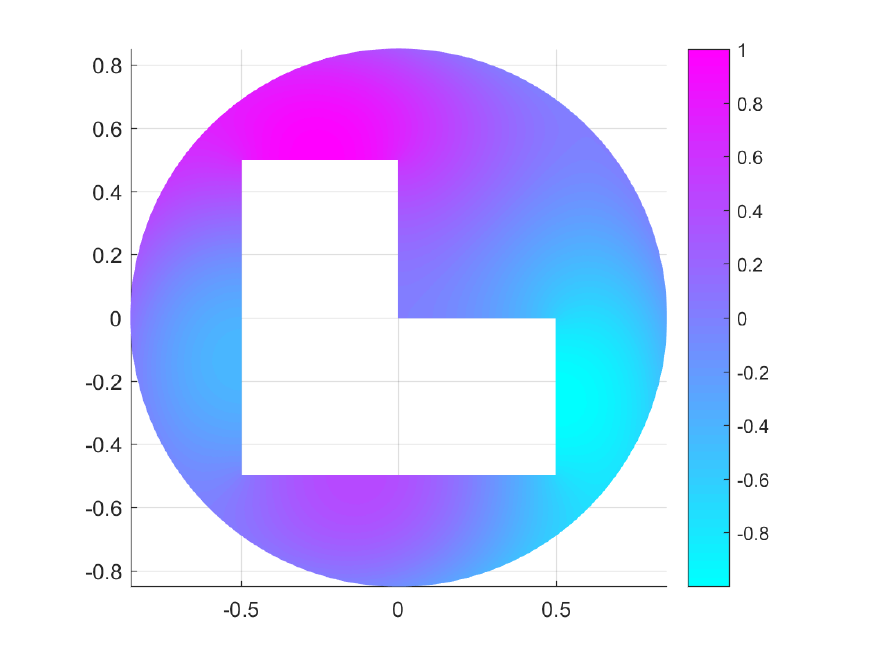}}
\end{tabular}
\end{center}
\caption{Real parts of the eigenfunctions for the smallest three scattering poles of the L-shaped domain.}
\label{fig-efs-3}
\end{figure}

% ------------------ Conclusions and Future Work ------------------
\section{Conclusions and Future Work}\label{CF}
We propose a finite element DtN method for the computation of scattering resonances of sound hard obstacles. The convergence is proved based on the abstract approximation theory for holomorphic Fredholm opeartor functions. Numerical results for three benchmark problems are presented.

It is highly nontrivial to develop effective numerical methods for scattering resonances. The numerical results indicate that  
\[
\text{finite element} + \text{DtN mapping} + \text{parallel SIM} 
\]
provides a practical framework, which can be used to compute scattering poles of other obstacles, inhomogeneous medium, and electromagnetics scattering problems.

\section*{Acknowledgement}
Y. Xi is partially supported by the National Natural Science Foundation of China with Grant No.11901295, No.12371440 and Natural Science Foundation of Jiangsu Province under BK20190431. B. Gong is supported by National Natural Science Foundation of China No.12201019. J. Sun is partially supported by an NSF Grant DMS-2109949 and a Simons Foundation Collaboration Grant 711922.

\bibliography{references}

\end{document}